\documentclass[12pt,a4paper]{amsart}
\allowdisplaybreaks[1]

\usepackage{amstext}
\usepackage{amsthm}
\usepackage{amssymb}

\newtheorem{thm}{Theorem}[section]
\newtheorem{lem}[thm]{Lemma}

\newtheorem{cor}[thm]{Corollary}

\theoremstyle{definition}
\newtheorem{defn}[thm]{Definition}

\theoremstyle{remark}
\newtheorem{rem}[thm]{Remark}

\begin{document}

\title[analytic properties of   Ohno function]{Analytic properties of  Ohno function}

\author{Ken Kamano}
\address[Ken Kamano]{Department of Robotics, Osaka Institute of Technology, 1-45 Chaya-machi, Kita-ku, Osaka 530-8568, Japan}
\email{ken.kamano@oit.ac.jp}

\author{Tomokazu Onozuka}
\address[Tomokazu Onozuka]{Institute of Mathematics for Industry, Kyushu University 744, Motooka, Nishi-ku,
Fukuoka, 819-0395, Japan}
\email{t-onozuka@imi.kyushu-u.ac.jp}


\begin{abstract}
Ohno's relation is a well-known relation on  the field of the multiple zeta values and  has an interpolation to complex function. 
In this paper, we call its complex function  Ohno function and study it.
 We consider the region of absolute convergence, give some new expressions, and  show new relations of the function. We also  give a direct proof of  the  interpolation of Ohno's relation.
\end{abstract}

\maketitle

\section{Introduction}
 For positive integers $k_1,\ldots,k_r$ with $k_r \ge 2$, the multiple zeta values (MZVs) are defined by
\begin{align*}
 \zeta(k_1,\ldots, k_r)
 :=\sum_{1\le n_1<\cdots <n_r} \frac {1}{n_1^{k_1}\cdots n_r^{k_r}}.
\end{align*}
We say that an index $(k_1,\ldots,k_r)\in\mathbb{Z}_{\ge1}^r$ is admissible if $k_r\ge2$.
For an index $\boldsymbol{k}=(k_1,\ldots,k_r)$, 
$\text{wt}(\boldsymbol{k}):=k_1+\cdots +k_r$ and $\text{dep}(\boldsymbol{k}):= r$ are called 
weight and depth of $\boldsymbol{k}$, respectively.
It is known that MZVs satisfy many algebraic relations over $\mathbb{Q}$. 
Ohno's relation is a well-known relation among MZVs.
\begin{defn}
For an admissible index 
 \[
  \boldsymbol{k}:=(\underbrace{1,\ldots,1}_{a_1-1},b_1+1,\dots,\underbrace{1,\ldots,1}_{a_d-1},b_d+1) \quad (a_i, b_i\ge1),
 \]
we define the dual index of $\boldsymbol{k}$ by 
 \[
  \boldsymbol{k}^\dagger :=(\underbrace{1,\ldots,1}_{b_d-1},a_d+1,\dots,\underbrace{1,\ldots,1}_{b_1-1},a_1+1).
 \]
\end{defn}
\begin{thm}[Ohno's relation; Ohno \cite{Oho99}] \label{ohno}
 For an admissible index $(k_1,\ldots,k_r)$ and $m\in\mathbb{Z}_{\ge 0}$, we have
 \begin{align*}
  \sum_{\substack{ e_1+\cdots+e_r=m \\ e_i\ge0\,(1\le i\le r) }}
  \zeta (k_1+e_1,\ldots,k_r+e_r) 
  =\sum_{\substack{ e'_1+\cdots+e'_{r'}=m \\ e'_i\ge0\,(1\le i\le r') }}
  \zeta (k'_1+e'_1,\ldots,k'_{r'}+e'_{r'}), 
 \end{align*}
 where the index $(k'_1,\ldots,k'_{r'})$ is the dual index of $(k_1,\ldots,k_r)$.
\end{thm}

For an admissible index $\boldsymbol{k}=(k_1,\ldots,k_r)$ and $s\in\mathbb{C}$ with $\Re(s)>-1$, Hirose-Murahara-Onozuka \cite{HMO20} defined the Ohno function $I_{\boldsymbol{k}}(s)$ by
\begin{align}\label{ohnofunc}
 I_{\boldsymbol{k}}(s):=\sum_{i=1}^{r}\sum_{0<n_1<\cdots<n_r} \frac{1}{n_1^{k_1}\cdots n_r^{k_r}}
  \cdot \frac{1}{n_i^{s}} \prod_{j\ne i} \frac{n_j}{ n_j-n_i }.
\end{align}
This is a sum of special cases of $\zeta_{\mathfrak{sl}(r+1)}(\boldsymbol{s})$ which is called the Witten multiple zeta function associated with $\mathfrak{sl}(r+1)$. 
The Witten multiple zeta  function was first introduced in Matsumoto-Tsumura \cite{MaTs06}, which is also called the zeta function associated with the root system of type $A_r$ (for more details, see Komori-Matsumoto-Tsumura \cite{KMT10}), and continued meromorphically to the whole complex space $\mathbb{C}^{r(r+1)/2}$. 
Hence $I_{\boldsymbol{k}}(s)$ can be continued meromorphically to $\mathbb{C}$.

When $s=m\in\mathbb{Z}_{\ge 0}$, the Ohno function is the Ohno sum, that is, 
\begin{align*}
I_{\boldsymbol{k}}(m)=\sum_{\substack{ e_1+\cdots+e_r=m \\ e_i\ge0\,(1\le i\le r) }}\zeta (k_1+e_1,\ldots,k_r+e_r),
\end{align*}
and by Theorem \ref{ohno}, we have
 \begin{align*}
  I_{\boldsymbol{k}}(m)=I_{\boldsymbol{k}^\dagger}(m).
 \end{align*}
In \cite{HMO20},  Hirose-Murahara-Onozuka gave an interpolation of Ohno's relation to complex function.
\begin{thm}[an interpolation of Ohno's relation] \label{interpolation}
 For an admissible index $\boldsymbol{k}$ and $s\in\mathbb{C}$, we have
 \begin{align*}
  I_{\boldsymbol{k}}(s)=I_{\boldsymbol{k}^\dagger}(s).
 \end{align*}
\end{thm}

In this paper, we study the Ohno function $I_{\boldsymbol{k}}(s)$. In  Section 2, we give a precise region of absolute convergence of the series \eqref{ohnofunc}.
\begin{thm}\label{main1}
The series \eqref{ohnofunc} converges absolutely  only for
\begin{align*}
\max_{1\leq j\leq r}\{r-2j+2-(k_j+\cdots+k_r)\}<\Re(s).
\end{align*}
\end{thm}
In Section 3, we give the following integral expression of the Ohno function;
\begin{thm}\label{thm:integral_exp}
For an admissible index 
 \[
  \boldsymbol{k}:=(\underbrace{1,\ldots,1}_{a_1-1},b_1+1,\dots,\underbrace{1,\ldots,1}_{a_d-1},b_d+1) \quad (a_i, b_i\ge1)
 \]
  and $s\in\mathbb{C}$ with $\Re(s)>-1$,  we have
\begin{align*}
I_{\boldsymbol{k}}(s) &=\dfrac{1}{(a_1-1)!(b_1-1)!  \cdots (a_d-1)! (b_d-1)! \Gamma (s+1) }\\
    &\times \int_{0<t_1<\cdots < t_{2d}<1}  \dfrac{dt_1 \cdots dt_{2d}}{(1-t_1)t_2  \cdots (1-t_{2d-1})t_{2d} }\\ 
	&\times \left(  \log  \dfrac{1-t_1}{1-t_2}   \right)^{a_1-1}
	\left(  \log \dfrac{t_3}{t_2}   \right)^{b_1-1}
	\cdots 
	\left( \log  \dfrac{1-t_{2d-1}}{1-t_{2d}}   \right)^{a_{d}-1}
	\left( \log  \dfrac{1}{t_{2d}}   \right)^{b_{d}-1} 
	\left( \log  \dfrac{t_2 \cdots t_{2d} }{t_1 \cdots t_{2d-1}}  \right)^{s}.
\end{align*}
\end{thm}
In Section 4, we give a new proof of Theorem \ref{interpolation}. In the original proof, we assume Ohno's relation, but our new proof  gives Theorem \ref{interpolation} directly. Our method is based on Theorem \ref{thm:integral_exp} and the proof given by Ulanskii \cite{U}. 
In Section 5, we consider an interpolation of $T$-interpolated sum formula.
Finally, in Section 6, we show another expression of the Ohno function. 

\begin{thm}\label{main6}
For an admissible index  $\boldsymbol{k}=(k_1,\ldots,k_r)$ and $s\in\mathbb{C}$ with $\max_{1\leq j\leq r}\{r-2j+2-(k_j+\cdots+k_r)\}<\Re(s)<0$,  we have
\begin{align*}
I_{\boldsymbol{k}}(s)=&-\frac{\sin(\pi s)}{\pi}\sum_{0<n_1<\cdots<n_r}\frac{1}{n_1^{k_1-1}\cdots n_r^{k_r-1}}\int_0^\infty\frac{w^{-s-1}}{(w+n_1)\cdots(w+n_r)}dw.
\end{align*}
\end{thm}

By applying this theorem, we can deduce linear relations among Ohno functions.

\begin{thm}\label{main7}
Let $l$ be a positive integer with $l \le r$ and an admissible index $\boldsymbol{k}=(k_1,\ldots,k_r)$  satisfy
$$\max_{\substack{1\leq j\leq r\\|\boldsymbol{e}|=m\\e_1,\ldots,e_{r}\leq1\\e_l=0}}\{r-2j+2-(k_j+e_j+\cdots+k_r+e_r)\}<-m$$
for all $m=0,\ldots,r-1$.
For   $s\in\mathbb{C}$, we have
\begin{align*}
\sum_{j=0}^{r-1}(-1)^{j}\sum_{\substack{|\boldsymbol{e}|=j\\e_1,\ldots,e_{r}\leq1\\e_l=0}}I_{\boldsymbol{k}+\boldsymbol{e}}(s-j)
=\zeta(k_1,\ldots,k_{l-1},k_l+s,k_{l+1},\ldots,k_r).
\end{align*}
\end{thm}

By  theorem \ref{main6}, we can deduce the following corollary.
\begin{cor}\label{cor}
For an admissible index $\boldsymbol{k}=(k_1,\ldots,k_r)$ and a negative integer $n$ with $\max_{1\leq j\leq r}\{r-2j+2-(k_j+\cdots+k_r)\}<n<0$, we have
\begin{align*}
I_{\boldsymbol{k}}(n)=0.
\end{align*}
\end{cor}

\section{Region of absolute convergence}

We give a precise region of absolute convergence of the series \eqref{ohnofunc}.
It is enough to consider a region of absolute convergence of the Dirichlet series
\begin{align*}
\begin{split}
  I_{\boldsymbol{k},i}(s):=&\sum_{0<n_1<\cdots<n_r} \frac{1}{n_1^{k_1}\cdots n_r^{k_r}}
  \cdot \frac{1}{n_i^{s}} \prod_{j\ne i} \frac{n_j}{ n_j-n_i }\\
    =& (-1)^{i-1} \sum_{m_1,\ldots,m_r=1}^\infty 
   \frac{1}{ m_1^{k_1-1}(m_1+m_2)^{k_2-1}\cdots(m_1+\cdots+m_r)^{k_r-1} } \\
   &\,\,\,\,\qquad \times \frac{1}{ (m_1+\cdots+m_i)^{s+1} } 
    \, \prod_{j<i} \frac{1}{m_{j+1}+\cdots+m_i }
    \, \prod_{j>i} \frac{1}{m_{i+1}+\cdots+m_j },
\end{split}
\end{align*}
for all $1\le i\le r$.
Since the above series is a special case of the zeta function associated with the root system of type $A_r$ for each $i$,
we can apply the result given by Zhao-Zhou \cite[Proposition 2.1]{ZhZh11}.
\begin{thm}[Zhao-Zhou \cite{ZhZh11}] \label{zhzh}
The series
\begin{align}\label{ZZ}
\sum_{m_1,\ldots,m_r=1}^\infty \prod_{\boldsymbol{i}\subseteq[r]}\left(\sum_{j=1}^{{\rm lg}(\boldsymbol{i})}m_{i_j}\right)^{-\sigma_{\boldsymbol{i}}}
\end{align}
converges if and only if for all $\ell=1,\ldots,r$ and $\boldsymbol{i}=(i_1,\ldots,i_\ell)\subseteq[r]$
\begin{align}\label{ZZ1}
\sum_{\substack{\boldsymbol{j}{\rm \;contains\;at\;least\;one\;of}\\ i_1,\ldots,i_\ell}}\sigma_{\boldsymbol{j}}>\ell,
\end{align}
where the product $\prod_{\boldsymbol{i}\subseteq[r]}$ runs over all nonempty subsets of $[r]=(1,2,\ldots,r)$ as a poset and ${\rm lg}(\boldsymbol{i})$ is the length of  $\boldsymbol{i}$.
\end{thm}
Note that if we put
\begin{align}\label{sigma}
\sigma_{\boldsymbol{i}}=
\begin{cases}
k_j-1&(\boldsymbol{i}=(1,2,\ldots,j)\mbox{ for } j\ne i),\\
k_i+s&(\boldsymbol{i}=(1,2,\ldots,i)),\\
1&(\boldsymbol{i}=(j+1,\ldots,i)\mbox{ for } j< i\mbox{ or }\boldsymbol{i}=(i+1\ldots,j)\mbox{ for } j> i),\\
0&(\mbox{otherwise}),
\end{cases}
\end{align}
then the series \eqref{ZZ} is $(-1)^{i-1}I_{\boldsymbol{k},i}(s)$.

\begin{proof}[Proof of Theorem \ref{main1}]

In the case $\ell=r$ for Theorem \ref{zhzh}, since $\boldsymbol{i}=(1,\ldots,r)$, the left-hand side of \eqref{ZZ1} is the sum of all $\sigma_{\boldsymbol{i}}$'s in \eqref{sigma}, so we obtain the following condition of absolute convergence
\begin{align}\label{cond1}
k_1+\cdots +k_r+\Re(s)>r.
\end{align}
When $\ell=r-1$, let us consider the case $\boldsymbol{i}=(2,\ldots,r)$ first. In this case, the left-hand side of \eqref{ZZ1} is the sum of all $\sigma_{\boldsymbol{i}}$'s but $\sigma_{(1)}$ in \eqref{sigma}, so we obtain the following conditions for $I_{\boldsymbol{k},i}(s)$:
\begin{align*}
\begin{cases}
k_2+\cdots +k_r+\Re(s)>r-2 &(i\neq1,\ \boldsymbol{i}=(2,\ldots,r)),\\
k_2+\cdots +k_r>r-1 &(i=1,\ \boldsymbol{i}=(2,\ldots,r)).
\end{cases}
\end{align*}
In the case $\boldsymbol{i}=(1,\ldots,i-1,i+1,\ldots,r)$ for $i\neq1$ or $\boldsymbol{i}=(1,\ldots,i,i+2,\ldots,r)$, the left-hand side of \eqref{ZZ1} is the sum of all $\sigma_{\boldsymbol{i}}$'s but $\sigma_{(i)}$ or $\sigma_{(i+1)}$ in \eqref{sigma}, respectively,  so we obtain the following condition:
\begin{align*}
k_1+\cdots +k_r+\Re(s)>r \quad(i\neq1,\ \boldsymbol{i}=(1,\ldots,i-1,i+1,\ldots,r))\ {\rm or}\ (\boldsymbol{i}=(1,\ldots,i,i+2,\ldots,r)).
\end{align*}
Otherwise, we  obtain the following condition:
\begin{align*}
k_1+\cdots +k_r+\Re(s)>r-1\quad(\rm{otherwise}).
\end{align*}
The conditions obtained by the case $\boldsymbol{i}\neq(2,\ldots,r)$ are contained in \eqref{cond1}, hence only the case $\boldsymbol{i}=(2,\ldots,r)$ deduces the region of absolute convergence when $\ell=r-1$.
(In general, it suffices to consider the case $\boldsymbol{i}=(j,j+1,\ldots,r)$ when $\ell=r-j+1$.)
By the similar way, considering all $\ell=1,\ldots,r-2$  for Theorem \ref{zhzh}, the series $I_{\boldsymbol{k},i}(s)$ converges absolutely only for
\begin{align*}
&k_1+\cdots +k_r+\Re(s)>r,\\
&k_2+\cdots +k_r+\Re(s)>r-2,\\
&\cdots\\
&k_i+\cdots+k_r+\Re(s)>r-2i+2
\end{align*}
and
\begin{align*}
&k_{i+1}+\cdots +k_r>r-i,\\
&k_{i+2}+\cdots +k_r>r-i-1,\\
&\cdots\\
&k_r>1.
\end{align*}
Hence, the series \eqref{ohnofunc} is absolutely convergent for
\begin{align*}
&k_1+\cdots +k_r+\Re(s)>r,\\
&k_2+\cdots +k_r+\Re(s)>r-2,\\
&\cdots\\
&k_r+\Re(s)>-r+2
\end{align*}
and
\begin{align*}
&k_{2}+\cdots +k_r>r-1,\\
&k_{3}+\cdots +k_r>r-2,\\
&\cdots\\
&k_r>1.
\end{align*}
Since the index $\boldsymbol{k}$ is admissible, we obtain Theorem \ref{main1}.
\end{proof}

\section{New integral expression}

We need the following lemma to prove Theorem \ref{thm:integral_exp}.
\begin{lem}\label{lamme:integral}
For $r\in \mathbb{Z}_{\ge 1}$, $c_i>0$ $(1\le i \le r)$ with $c_i\neq c_j$ $(i\neq j)$ and $s\in\mathbb{C}$ with $\Re (s) >-1$,
we have	\begin{equation}\label{eq:int_lemma}
\begin{split}
		& \sum_{i=1}^r  \left(   \dfrac{1}{c_i^{s+1}}  \prod_{j\neq i}  \dfrac{1}{c_j-c_i}   \right) \\
		&= \dfrac{1}{\Gamma(s+1)} \int_{0}^{\infty}  \cdots \int_{0}^{\infty}  
		e^{-c_1 x_1-  \cdots -c_r x_r} (x_1+\cdots +x_r)^{s} \, dx_1 \cdots dx_r .
\end{split}	\end{equation}
\end{lem}

\begin{proof}
First we prove the integral in the right-hand side of \eqref{eq:int_lemma}
converges for $\Re(s)>-1$.
Let $\sigma:=\Re(s)$.
When $-1<\sigma<0$, we have
\begin{align*}
&\left| \int_{0}^{\infty}  \cdots \int_{0}^{\infty}  
e^{-c_1 x_1-  \cdots -c_r x_r} (x_1+\cdots +x_r)^{s} \, dx_1 \cdots dx_r \right|\\
&\le  \int_{0}^{\infty}  \cdots \int_{0}^{\infty}  
e^{-c_1 x_1-  \cdots -c_r x_r} (x_1+\cdots +x_r)^{\sigma} \, dx_1 \cdots dx_r \\
&\le  \int_{0}^{\infty}  \cdots \int_{0}^{\infty}  
e^{-c_1 x_1-  \cdots -c_r x_r} x_r^{\sigma} \, dx_1 \cdots dx_r \\
&\le  \dfrac{1}{c_1\cdots c_{r-1}}  \int_0^{\infty}
e^{-c_r x_r} x_r^{\sigma} \,  dx_r \\
&= \dfrac{1}{c_1\cdots c_{r-1}}  \dfrac{1}{c_r^{\sigma+1}}
\Gamma(\sigma +1).
\end{align*}
Hence the integral in the right-hand side of \eqref{eq:int_lemma} converges.
When $\sigma \ge 0$, 
we have
\begin{align*}
	&\left| \int_{0}^{\infty}  \cdots \int_{0}^{\infty}  
	e^{-c_1 x_1-  \cdots -c_r x_r} (x_1+\cdots +x_r)^{s} \, dx_1 \cdots dx_r \right|\\
	&\le \int_{0}^{\infty}  \cdots \int_{0}^{\infty}  
e^{-c_1 x_1-  \cdots -c_r x_r} (x_1+1)^{\sigma}\cdots (x_r+1)^{\sigma} \, dx_1 \cdots dx_r \\
	&= e^{c_1} \cdots e^{c_r} \int_{1}^{\infty}  \cdots \int_{1}^{\infty}  
e^{-c_1 x_1-  \cdots -c_r x_r} x_1^{\sigma}\cdots x_r^{\sigma} \, dx_1 \cdots dx_r \\
	&\le e^{c_1} \cdots e^{c_r} \dfrac{1}{c_1^{\sigma+1}\cdots c_r^{\sigma+1}} 
	\Gamma(\sigma+1)^r.
\end{align*}
Hence the integral in the right-hand side of \eqref{eq:int_lemma}
also converges in this case.

Next we prove the equation \eqref{eq:int_lemma} by induction on $r$.
When $r=1$, the right-hand side of  \eqref{eq:int_lemma} equals
\begin{align*}
\dfrac{1}{\Gamma(s+1)}\int_0^{\infty}e^{-c_1x} x^s\, dx
= \dfrac{1}{c_1^{s+1}}\end{align*}
and the equation \eqref{eq:int_lemma} holds.
Suppose $r\ge 2$ and
\eqref{eq:int_lemma} holds for $r-1$ variables.
By the change of variables 
\begin{align}
\begin{cases}\label{CoV}
x_1+\cdots +x_r=y_1, \\ 
x_2+\cdots +x_r=y_2, \\ 
\vdots  \\
x_r=y_r,
\end{cases}
\end{align}
we have
\begin{align*}
& \int_{0}^{\infty}  \cdots \int_{0}^{\infty}  
e^{-c_1 x_1-  \cdots -c_r x_r} (x_1+\cdots +x_r)^{s} \, dx_1 \cdots dx_r \\
&=
 \int \cdots \int_{y_1>\cdots >y_r\ge0}  
e^{-c_1 y_1} e^{(c_1-c_2)y_2}  \cdots   e^{(c_{r-1}-c_r)y_r} y_1^{s} \, dy_1 \cdots dy_r \\
&=
\int \cdots \int_{y_1>\cdots >y_{r-1}>0}  
e^{-c_1 y_1} e^{(c_1-c_2)y_2}  \cdots  e^{(c_{r-2}-c_{r-1})y_{r-1}} \dfrac{e^{(c_{r-1}-c_r)y_{r-1}}-1}{c_{r-1}-c_r} y_1^{s} \, dy_1 \cdots dy_{r-1} \\
&=
\int \cdots \int_{y_1>\cdots >y_{r-1}>0}  
e^{-c_1 y_1} e^{(c_1-c_2)y_2}  \cdots  e^{(c_{r-3}-c_{r-2})y_{r-2}} \dfrac{e^{(c_{r-2}-c_r)y_{r-1}} - e^{(c_{r-2}-c_{r-1})y_{r-1}}}{c_{r-1}-c_r} 
 y_1^{s} \, dy_1 \cdots dy_{r-1}.
\end{align*}
By the change of variables \eqref{CoV} with $x_r=y_r=0$, we have
\begin{align*}
&\int \cdots \int_{y_1>\cdots >y_{r-1}>0}  
e^{-c_1 y_1} e^{(c_1-c_2)y_2}  \cdots  e^{(c_{r-3}-c_{r-2})y_{r-2}} \dfrac{e^{(c_{r-2}-c_r)y_{r-1}} - e^{(c_{r-2}-c_{r-1})y_{r-1}}}{c_{r-1}-c_r} 
 y_1^{s} \, dy_1 \cdots dy_{r-1}\\
 &=\dfrac{1}{c_{r-1}-c_r} \int_{0}^{\infty}  \cdots \int_{0}^{\infty}  
e^{-c_1 x_1-  \cdots -c_{r-2} x_{r-2}-c_r x_{r-1}} (x_1+\cdots +x_{r-1})^{s} \, dx_1 \cdots dx_{r-1} \\
&\quad-\dfrac{1}{c_{r-1}-c_r} \int_{0}^{\infty}  \cdots \int_{0}^{\infty}  
e^{-c_1 x_1-  \cdots -c_{r-2} x_{r-2}-c_{r-1} x_{r-1}} (x_1+\cdots +x_{r-1})^{s} \, dx_1 \cdots dx_{r-1} .
\end{align*}
By the induction hypothesis, this equals
\begin{align*}
&\dfrac{\Gamma(s+1)}{c_{r-1}-c_r}
\Bigg(  \sum_{i=1}^{r-2} \left(  \frac{1}{c_i^{s+1}}  \frac{1}{c_r-c_i} \prod_{ \substack{j\neq i \\ 1\le j \le r-2} } \frac{1}{c_j-c_i}  \right)
+  \frac{1}{c_r^{s+1}} \prod_{ \substack{ 1\le j \le r-2} } \frac{1}{c_j-c_r} \\
&-  \sum_{i=1}^{r-2} \left(  \frac{1}{c_i^{s+1}}  \frac{1}{c_{r-1}-c_i} \prod_{ \substack{j\neq i \\ 1\le j \le r-2} } \frac{1}{c_j-c_i}  \right)
-   \frac{1}{c_{r-1}^{s+1}} \prod_{ \substack{ 1\le j \le r-2} } \frac{1}{c_j-c_{r-1}}
\Bigg) \\
&= \Gamma(s+1) \left( \sum_{i=1}^{r-2} \frac{1}{c_i^{s+1}}  \prod_{ \substack{j\neq i \\ 1\le j \le r} } \frac{1}{c_j-c_i} 
+ \frac{1}{c_r^{s+1}}  \prod_{ \substack{j\neq r \\ 1\le j \le r-1} } \frac{1}{c_j-c_r} 
+ \frac{1}{c_{r-1}^{s+1}}  \prod_{ \substack{j\neq r-1 \\ 1\le j \le r} } \frac{1}{c_j-c_{r-1}} 
\right) \\
&= \Gamma(s+1) \sum_{i=1}^{r} \frac{1}{c_i^{s+1}}  \prod_{ \substack{j\neq i \\ 1\le j \le r} } \frac{1}{c_j-c_i}
\end{align*}
and this proves that \eqref{eq:int_lemma} holds for $r$.
\end{proof}

\begin{proof}[Proof of Theorem \ref{thm:integral_exp}]
	We set 
	 \[
  \boldsymbol{k}=(\underbrace{1,\ldots,1}_{a_1-1},b_1+1,\dots,\underbrace{1,\ldots,1}_{a_d-1},b_d+1) = (k_1,\ldots , k_r).
 \]
Note that $a_1+\cdots +a_d=r$.
	By definition, we have
	\begin{align*}
		I_{\boldsymbol{k}}(s)=
		\sum_{0<n_1< \cdots <n_r} 
		\dfrac{1}{n_1^{k_1-1} n_2^{k_2-1} \cdots n_r^{k_r-1} }
		\sum_{i=1}^r \left( \dfrac{1}{n_i^{s+1}}    \prod_{ j\neq i }\dfrac{1}{n_j- n_i } \right).
	\end{align*}
	Because
	$(k_1-1, \ldots ,k_r-1)=
	( \underbrace{0,\ldots ,0}_{a_1-1}, b_1, \ldots ,  \underbrace{0,\ldots ,0}_{a_d-1}, b_d )$, 
	we have
	\begin{align*}
		I_{\boldsymbol{k}}(s)=
		\sum_{0<n_1< \cdots <n_r} 
		\dfrac{1}{n_{a_1}^{b_1} n_{a_1+a_2}^{b_2} \cdots n_{a_1+\cdots + a_d}^{b_d} }
		\sum_{i=1}^r \left( \dfrac{1}{n_i^{s+1}}    \prod_{ j\neq i }\dfrac{1}{n_j- n_i } \right).
	\end{align*}
	By using the well-known identity
	\[ \displaystyle \dfrac{1}{n^b} = \dfrac{1}{\Gamma(b)} \int_0^{\infty} e^{-ny} y^{b-1}dy\ \ \ (n, \, b>0)\]
	and Lemma \ref{lamme:integral}, we get 
	\begin{align*} 
		I_{\boldsymbol{k}}(s)
		&= \dfrac{1}{\Gamma(s+1) (b_1-1)! \cdots (b_d-1)!} \\
		&\sum_{0<n_1< \cdots <n_r}
		\int_0^{\infty} \cdots \int_0^{\infty} e^{-n_{a_1}y_1} e^{-n_{a_1+a_2}y_2}\cdots e^{-n_{a_1+\cdots +a_d}y_d}
		y_1^{b_1-1} \cdots y_d^{b_d-1}\,dy_1\cdots dy_d\\
		&\hspace{110pt}e^{-n_1x_1}  e^{-n_2x_2}\cdots e^{-n_{r}x_r}(x_1+\cdots +x_r)^s \,dx_1\cdots dx_r.
	\end{align*}
	Set $n_i= m_1+\cdots +m_i$ $(1\le i \le r)$.
	Then each $m_i$ runs over all positive integers and
	\begin{align*} 
		I_{\boldsymbol{k}}(s)
		=& \dfrac{1}{\Gamma(s+1) (b_1-1)! \cdots (b_d-1)!} \\
		&\sum_{m_1\ge 1, \ldots ,  m_r\ge 1}
		\int_0^{\infty} \cdots \int_0^{\infty} 
		\exp\left(-\sum_{i=1}^{r}m_i(X_i+Y_i)\right)\\
		&\hspace{40pt} (x_1+\cdots +x_r)^s y_1^{b_1-1} \cdots y_d^{b_d-1}
		 \,dx_1\cdots dx_r\,dy_1\cdots dy_d\\
		=& \dfrac{1}{\Gamma(s+1) (b_1-1)! \cdots (b_d-1)!} \\
&\int_0^{\infty} \cdots \int_0^{\infty} 
\prod_{i=1}^r \dfrac{e^{-X_i-Y_i}}{1-e^{-X_i-Y_i}}\\
&\hspace{40pt} (x_1+\cdots +x_r)^sy_1^{b_1-1} \cdots y_d^{b_d-1}
\,dx_1\cdots dx_r\,dy_1\cdots dy_d,
	\end{align*}
where
\begin{align*} 
X_i&=x_i+\cdots+x_r\ \ (1\le i \le r),\\
Y_i&=\begin{cases}
      y_1+\cdots +y_d &(1\le i \le a_1),\\
      y_2+\cdots +y_d &(a_1< i \le a_1+a_2),\\
      \vdots\\
      y_d &(a_1+\cdots +a_{d-1}< i \le r).
  	\end{cases}
\end{align*} 

We apply the following change of variables:
\[ \begin{cases}
	t^{(1)}_{1}&= \exp(-X_1-Y_{a_1}),\\
	\vdots &\hspace{20pt}\vdots \\
	t^{(1)}_{a_1}&= \exp(-X_{a_1}-Y_{a_1}),\\
	u_1&= \exp(-X_{a_1+1}-Y_{a_1}),\end{cases}\ \ \ \ \ 
\begin{cases}
t^{(2)}_{1}&= \exp(-X_{a_1+1}-Y_{a_1+a_2}),\\
\vdots &\hspace{20pt}\vdots \\
t^{(2)}_{a_2}&= \exp(-X_{a_1+a_2}-Y_{a_1+a_2}),\\
u_2&= \exp(-X_{a_1+a_2+1}-Y_{a_1+a_2}),\end{cases}\] 

\[ \hspace{20pt} \cdots 
\begin{cases}
t^{(d)}_{1}&= \exp(-X_{a_1+\cdots +a_{d-1}+1}-Y_r),\\
\vdots &\hspace{20pt}\vdots \\
t^{(d)}_{a_d}&= \exp(-X_{r}-Y_r),\\
u_d&= \exp(-Y_r).\end{cases}\]

Then it can be easily checked that 
\begin{itemize}
\item $0<t^{(1)}_{1}<\cdots <t^{(1)}_{a_1}<u_1 < \cdots  \cdots < t^{(d)}_{1} <\cdots < t^{(d)}_{a_d} < u_d<1$,
\item $\displaystyle \prod_{i=1}^r \dfrac{e^{-X_i-Y_i}}{1-e^{-X_i-Y_i}}
=\dfrac{t_1^{(1)}}{1-t_1^{(1)}} \cdots 
\dfrac{t_{a_d}^{(d)}}{1-t_{a_d}^{(d)}}$,
\item  $y_1=\log \dfrac{t_1^{(2)}}{u_1},\ \  \ldots ,\ \ 
y_{d-1}=\log \dfrac{t_1^{(d)}}{u_{d-1}},\ \  
y_d=\log \dfrac{1}{u_d}$, 
\item $x_1+\cdots +x_r= \log \dfrac{u_1}{t_1^{(1)}}
\dfrac{u_2}{t_1^{(2)}} \cdots \dfrac{u_d}{t_1^{(d)}}$,
\item $ dx_1 \cdots dx_r dy_1\cdots dy_d = \dfrac{1}{t^{(1)}_{1}\cdots t^{(d)}_{a_d} u_1 \cdots u_d} dt^{(1)}_{1}\cdots dt^{(d)}_{a_d} du_1 \cdots du_d$.
\end{itemize}

Therefore 
	\begin{align*} 
		I_{\boldsymbol{k}}(s)
		=& \dfrac{1}{\Gamma(s+1) (b_1-1)! \cdots (b_d-1)!} 
		\int_{0<t^{(1)}_{1}<\cdots <t^{(1)}_{a_1}<u_1 < \cdots < t^{(d)}_{1} <\cdots < t^{(d)}_{a_d} < u_d<1}\\
		&\dfrac{dt^{(1)}_{1}\cdots dt^{(d)}_{a_d} du_1 \cdots du_d}{ (1-t^{(1)}_{1}) \cdots (1-t^{(1)}_{a_1}) u_1  \cdots (1-t^{(d)}_{1}) \cdots 
			(1-t^{(1)}_{a_d})u_d  }\\
		&\left(  \log \dfrac{t^{(2)}_{1}}{u_1}   \right)^{b_1-1}  \cdots \left(  \log \dfrac{t^{(d)}_{1}}{u_{d-1}}   \right)^{b_{d-1}-1}
		\left(  \log \dfrac{1}{u_{d}}   \right)^{b_{d}-1}
		\left(  \log \dfrac{u_1}{t^{(1)}_{1}} \dfrac{u_2}{t^{(2)}_{1}}   \cdots \dfrac{u_d}{t^{(d)}_{1}}   \right)^{s}
		\\
		=& \dfrac{1}{\Gamma(s+1)  (a_1-1)! \cdots (a_d-1)! (b_1-1)! \cdots (b_d-1)!} 
		\int_{0<t_1<u_1 < \cdots < t_{d}< u_d<1}\\
		&\dfrac{dt_1\cdots dt_d du_1\cdots du_d}{ (1-t_{1})  u_1  \cdots (1-t_d) u_d  }\\
		& \left(  \log \dfrac{1-t_{1}}{1-u_1}   \right)^{a_1-1} \cdots  \left(  \log \dfrac{1-t_{d}}{1-u_d}   \right)^{a_d-1} \\
		&\left(  \log \dfrac{t_{2}}{u_1}   \right)^{b_1-1}  \cdots \left(  \log \dfrac{t_{d}}{u_{d-1}}   \right)^{b_{d-1}-1}
		\left(  \log \dfrac{1}{u_{d}}   \right)^{b_{d}-1}
		\left(  \log \dfrac{u_1}{t_{1}} \dfrac{u_2}{t_{2}}   \cdots \dfrac{u_d}{t_{d}}   \right)^{s}.
	\end{align*}
	Here $t^{(i)}_1$ is replaced by $t_i$ in the last equation.
	This completes the proof.
\end{proof}

\section{New Proof of Theorem \ref{interpolation}}
In this section, we give another proof of $I_{\boldsymbol{k}}(s) =I_{\boldsymbol{k^{\dagger}}}(s)$
by using Theorem \ref{thm:integral_exp}.

Let $d$ be a positive integer. We consider the following change of variables:
	\begin{align*}
	\dfrac{1-t_{2\ell-1}}{1-t_{2\ell}} = \dfrac{u_{2(d-\ell+1)+1}} {u_{2(d-\ell+1)}}, \ \ \ 
	\dfrac{t_{2\ell}}{t_{2\ell+1}} = \dfrac{1-u_{2(d-\ell+1)}} {1-u_{2(d-\ell+1)-1}}\ \ \ \ \ (1\le \ell \le d).
	\end{align*}
Here we set $u_{2d+1}=t_{2d+1}=1$.

\begin{rem}
The change of variables above is a special case of that of Ulanskii \cite{U}.
By using this change of variables, he gave a direct proof of Ohno's relation for MZVs.
\end{rem}
This change of variables satisfies the following properties (cf. \cite[Sect.~2]{U}): 	
	\begin{enumerate}
	\item the region $0<t_1 < \cdots <t_{2d}<1$ corresponds to
the region $0<u_1 < \cdots <u_{2d}<1$,
	
	\item $\dfrac{dt_1 \cdots dt_{2d}}{(1-t_1)t_2  \cdots (1-t_{2d-1})t_{2d} } = 
	\dfrac{du_1 \cdots du_{2d}}{(1-u_1)u_2  \cdots (1-u_{2d-1})u_{2d} }$,
	
	\item $\dfrac{t_2}{t_1} \cdots \dfrac{t_{2d}}{t_{2d-1}} = \dfrac{u_2}{u_1} \cdots \dfrac{u_{2d}}{u_{2d-1}}$.
\end{enumerate}
	
By this change of variables, we have
	\begin{align*}
	&\int_{0<t_1<\cdots < t_{2d}<1}  \dfrac{dt_1 \cdots dt_{2d}}{(1-t_1)t_2  \cdots (1-t_{2d-1})t_{2d} }\\ 
	&\times \left(  \log  \dfrac{1-t_1}{1-t_2}   \right)^{a_1-1}
	\left(  \log \dfrac{t_3}{t_2}   \right)^{b_1-1}
	\cdots 
	\left( \log  \dfrac{1-t_{2d-1}}{1-t_{2d}}   \right)^{a_{d}-1}
	\left( \log  \dfrac{1}{t_{2d}}   \right)^{b_{d}-1} 
	\left( \log  \dfrac{t_2 \cdots t_{2d} }{t_1 \cdots t_{2d-1}}  \right)^{s} 
	\\
	=& \int_{0<u_1<\cdots < u_{2d}<1}  \dfrac{du_1 \cdots du_{2d}}{(1-u_1)u_2  \cdots (1-u_{2d-1})u_{2d} }\\ 
	&\hspace{20pt}\times
	\left(  \log \dfrac{1}{u_{2d}}   \right)^{a_1-1}  
	\left(  \log \dfrac{1-u_{2d-1}}{1-u_{2d}}   \right)^{b_1-1}
	\cdots 
	\left( \log  \dfrac{u_3}{u_2}   \right)^{a_{d}-1}
	\left(  \log \dfrac{1-u_1}{1-u_2}   \right)^{b_{d}-1}
	\left( \log  \dfrac{u_2 \cdots u_{2d} }{u_1 \cdots u_{2d-1}}\ \right)^{s}.
	\end{align*}
By multiplying the gamma factors, we have
$I_{\boldsymbol{k}}(s) =I_{\boldsymbol{k^{\dagger}}}(s)$.

\section{$T$-interpolation of Ohno function}

For an admissible index $\boldsymbol{k}=(k_1,\ldots ,k_r)$,
Yamamoto \cite{Y} introduced an
interpolated multiple zeta value $\zeta^T(\boldsymbol{k})$ as
\begin{align*}
\zeta^T(\boldsymbol{k}) = 
\sum_{\boldsymbol{p}} T^{r-\text{dep}(\boldsymbol{p})} 
\zeta(\boldsymbol{p}) \ \in \mathbb{R}[T].
\end{align*}
Here the sum runs over all indices $\boldsymbol{p}$ 
such that 
\[\boldsymbol{p} =(k_1 \square k_2 \square \cdots \square k_r), \]
where each $\square$ is filled by the comma $,$ or the plus $+$.
This polynomial in $T$ interpolates two kinds of multiple zeta values,
i.e., 
$\zeta^{0}(\boldsymbol{k}) =\zeta(\boldsymbol{k})$ and 
$\zeta^{1}(\boldsymbol{k}) =\zeta^{\star}(\boldsymbol{k})$,
where $\zeta^{\star}(\boldsymbol{k})$ is the multiple zeta-star value
defined by
\begin{align*}
	\zeta^{\star}(k_1,\ldots, k_r)
	:=\sum_{1\le n_1\le \cdots \le n_r} \frac {1}{n_1^{k_1}\cdots n_r^{k_r}}.
\end{align*}
Yamamoto proved the following sum formula for $\zeta^T(\boldsymbol{k})$,
which is an interpolation of the sum formulas 
for multiple zeta and zeta-star values.
\begin{thm}[{\cite[Theorem 1.1]{Y}}]\label{thm:t-sum}
For integers $m$, $a \ge 1$, we have
	\begin{align*}
		\sum_{\substack{ {\rm wt}(\boldsymbol{k})=m+a+1 \\  \boldsymbol{k}:{\,\rm admissible},\,{\rm dep}(\boldsymbol{k})=a }} \zeta^T(\boldsymbol{k})
		&= \sum_{j=0}^{a-1} \binom{m+a}{j} T^j (1-T)^{a-1-j} \zeta(m+a+1) \\
		&= \sum_{i=0}^{a-1} \binom{m+i}{i} T^i  \zeta(m+a+1).
	\end{align*}
\end{thm}

For $a\in \mathbb{Z}_{\ge 1}$,
let $\boldsymbol{a} = ( \underbrace{1,\ldots ,1}_{a-1}, 2)$.
We define an interpolated version of $I_{\boldsymbol{a}}(s)$ as
\begin{align*}
	I^T_{\boldsymbol{a}}(s) &:=\dfrac{1}{(a-1)! \Gamma (s+1) }\\
	&\times \int_{0<t_1<t_{2}<1}  \dfrac{dt_1  dt_{2}}{(1-t_1)t_2 } \left(  \log  \dfrac{1-t_1}{1-t_2}  + T \log \dfrac{t_2}{t_1}  \right)^{a-1}
	\left( \log  \dfrac{t_2 }{t_1 }  \right)^{s}\ \ \ (\Re(s)>-1). \end{align*}
When $T=0$, we have $I^0_{\boldsymbol{a}}(s) = I_{\boldsymbol{a}}(s)$.
When $s=m \in \mathbb{Z}_{\ge 0}$, the value
$I^T_{\boldsymbol{a}}(m)$ is the sum of all interpolated multiple zeta values
for fixed weight and depth:
\[ I^T_{\boldsymbol{a}}(m) = 
	\sum_{|\boldsymbol{e}|=m} \zeta^{T}(\boldsymbol{a}\oplus \boldsymbol{e})
	=\sum_{\substack{ {\rm wt}(\boldsymbol{k})=m+a+1 \\  \boldsymbol{k}: \text{admissible}, \text{dep} (\boldsymbol{k})=a }} \zeta^T(\boldsymbol{k}).
	\]

We can give the following formula, 
which is an interpolation of 
the sum formula for $T$-interpolated multiple zeta values.
In fact, this theorem deduces Theorem \ref{thm:t-sum} by setting $s=m$.
Since the proof is same as that in the last section, we omit it.
\begin{thm}
For $s\in \mathbb{C}$ and $a\in\mathbb{Z}_{\ge 1}$, we have
	\begin{align*}
		I^T_{\boldsymbol{a}}(s)=
		\left(  \sum_{i=0}^{a-1}  \binom{s+i}{i} T^i  \right) \zeta(s+a+1).
	\end{align*}
\end{thm}

\section{New Relations}
We first prove Theorem \ref{main6}.

\begin{proof}[Proof of Theorem \ref{main6}]
By the partial fraction decomposition, we have
\begin{align*}
\frac{1}{(w+n_1)\cdots(w+n_r)}=\sum_{i=1}^r \frac{1}{w+n_i}\cdot\prod_{j\ne i} \frac{1}{ n_j-n_i }.
\end{align*}
Let $B(x,y)$ be the beta function. Since
$$
\int_0^\infty\frac{w^{-s}}{w+n}dw=n^{-s}B(s,1-s)=n^{-s}\frac{\pi}{\sin(\pi s)}
$$
for $0<\Re(s)<1$, we have
\begin{align*}
\int_0^\infty\frac{w^{-s-1}}{(w+n_1)\cdots(w+n_r)}dw=-\frac{\pi}{\sin(\pi s)}\sum_{i=1}^r \frac{1}{n_i^{s+1}}\prod_{j\ne i} \frac{1}{ n_j-n_i }.
\end{align*}
Therefore, the statement of theorem holds for $-1<\Re(s)<0$.

For $-r<\Re(s)<0$ and $0\le i\le r$, we have
\begin{align*}
\int_{n_i}^{n_{i+1}}\frac{w^{-\sigma-1}}{(w+n_1)\cdots(w+n_r)}dw&\le \frac{1}{n_{i+1}n_{i+2}\cdots n_r}\int_{n_i}^{n_{i+1}}w^{-\sigma-i-1}dw\\
&\ll \begin{cases}
(n_{i+1}n_{i+2}\cdots n_r)^{-1}n_{i+1}^{-\sigma-i}&(-\sigma>i),\\
(n_{i+1}n_{i+2}\cdots n_r)^{-1}\log n_{i+1}&(-\sigma=i),\\
(n_{i+1}n_{i+2}\cdots n_r)^{-1}n_{i}^{-\sigma-i}&(-\sigma<i),
\end{cases}\\
&\ll \begin{cases}
(n_{i+1}n_{i+2}\cdots n_r)^{-1}n_{i+1}^{-\sigma-i}\log (n_{i+1}+1)&(-\sigma\ge i),\\
(n_{i+1}n_{i+2}\cdots n_r)^{-1}n_{i}^{-\sigma-i}&(-\sigma<i),
\end{cases}
\end{align*}
where $n_0=0$ and $n_{r+1}=\infty$. Hence we can estimate
\begin{align}\label{bunkatsu}
\begin{split}
&\sum_{0<n_1<\cdots<n_r}\frac{1}{n_1^{k_1-1}\cdots n_r^{k_r-1}}\int_{0}^{\infty}\frac{w^{-\sigma-1}}{(w+n_1)\cdots(w+n_r)}dw\\
&\ll
\sum_{0\le i\le-\sigma}\sum_{0<n_1<\cdots<n_r}(n_1^{-k_1+1}\cdots n_i^{-k_i+1})n_{i+1}^{-k_{i+1}-\sigma-i}(n_{i+2}^{-k_{i+2}}\cdots n_r^{-k_r})\log (n_{i+1}+1)\\
&\quad+\sum_{-\sigma<i\le r}\sum_{0<n_1<\cdots<n_r}(n_1^{-k_1+1}\cdots n_{i-1}^{-k_{i-1}+1})n_{i}^{-k_{i}-\sigma-i+1}(n_{i+1}^{-k_{i+1}}\cdots n_r^{-k_r}).
\end{split}
\end{align}
Note that by Matsumoto \cite{Mat02}, the infinite series $\sum_{0< n_1<\cdots <n_r} n_1^{-\sigma_1}\cdots n_r^{-\sigma_r}$ converges for
\begin{align*}
&\sigma_{1}+\cdots +\sigma_r>r,\\
&\sigma_{2}+\cdots +\sigma_r>r-1,\\
&\cdots\\
&\sigma_r>1.
\end{align*}
Thus the first  series in the right-hand side of \eqref{bunkatsu} converges when
\begin{align*}
&k_{1}+\cdots +k_r+\sigma>r,\\
&k_{2}+\cdots +k_r+\sigma>r-2,\\
&\cdots\\
&k_{i+1}+\cdots +k_r+\sigma>r-2i
\end{align*}
for $0\le i\le-\sigma$.
These conditions can be simply written as 
\begin{align*}
\max_{1\leq j\leq 1-\sigma}\{r-2j+2-(k_j+\cdots+k_r)\}<\sigma.
\end{align*}
Similarly   the second series of the right-hand side in \eqref{bunkatsu} converges when 
\begin{align*}
\max_{1\leq j\leq r}\{r-2j+2-(k_j+\cdots+k_r)\}<\sigma.
\end{align*}
Hence by the identity theorem, we obtain the theorem.

\end{proof}

\begin{proof}[Proof of Theorem \ref{main7}]
Let $E_j$ be the elementary symmetric polynomial of degree $j$ in\\
 $(n_1,\ldots,n_{\ell-1},n_{\ell+1},\ldots,n_{r})$. Then we have
\begin{align}\label{symmetric}
\begin{split}
&-\frac{\sin(\pi s)}{\pi}\sum_{0<n_1<\cdots<n_r}\frac{n_\ell}{n_1^{k_1}\cdots n_r^{k_r}}\int_0^\infty\frac{w^{-s-1}(w+n_1)\cdots(w+n_{\ell-1})(w+n_{\ell+1})\cdots(w+n_{r})}{(w+n_1)\cdots(w+n_r)}dw\\
&=-\frac{\sin(\pi s)}{\pi}\sum_{0<n_1<\cdots<n_r}\frac{n_\ell}{n_1^{k_1}\cdots n_r^{k_r}}\int_0^\infty\frac{w^{-s-1}(w^{r-1}+E_1w^{r-2}+\cdots+E_{r-1})}{(w+n_1)\cdots(w+n_r)}dw.
\end{split}
\end{align}
The left-hand side in \eqref{symmetric} is $\zeta(k_1,\ldots,k_{\ell-1},k_\ell+s,k_{\ell+1},\ldots,k_r)$ for $-1<\sigma<0$. On the other hand, the right-hand side in \eqref{symmetric} is 
\begin{align*}
&(-1)^{r-1}\sum_{\substack{|\boldsymbol{e}|=r-1\\e_1,\ldots,e_{r}\leq1\\e_\ell=0}}I_{\boldsymbol{k}+\boldsymbol{e}}(s-r+1)+(-1)^{r-2}\sum_{\substack{|\boldsymbol{e}|=r-2\\e_1,\ldots,e_{r}\leq1\\e_\ell=0}}I_{\boldsymbol{k}+\boldsymbol{e}}(s-r+2)\\
&+(-1)^{r-3}\sum_{\substack{|\boldsymbol{e}|=r-3\\e_1,\ldots,e_{r}\leq1\\e_\ell=0}}I_{\boldsymbol{k}+\boldsymbol{e}}(s-r+3)\\
&+\cdots\\
&+\sum_{\substack{|\boldsymbol{e}|=0\\e_1,\ldots,e_{r}\leq1\\e_\ell=0}}I_{\boldsymbol{k}+\boldsymbol{e}}(s)\\
=&\sum_{m=0}^{r-1}(-1)^{m}\sum_{\substack{|\boldsymbol{e}|=m\\e_1,\ldots,e_{r}\leq1\\e_\ell=0}}I_{\boldsymbol{k}+\boldsymbol{e}}(s-m).
\end{align*}
for 
\begin{align}\label{last}
\max_{\substack{1\leq j\leq r\\|\boldsymbol{e}|=m\\e_1,\ldots,e_{r}\leq1\\e_\ell=0}}\{r-2j+2-(k_j+e_j+\cdots+k_r+e_r)\}<\Re(s-m)<0\quad(m=0,\ldots,r-1).
\end{align}
If $\boldsymbol{k}$ satisfies 
$$\max_{\substack{1\leq j\leq r\\|\boldsymbol{e}|=m\\e_1,\ldots,e_{r}\leq1\\e_\ell=0}}\{r-2j+2-(k_j+e_j+\cdots+k_r+e_r)\}<-m$$
for all $m=0,\ldots,r-1$, the inequality \eqref{last} holds for $-1<\sigma<0$. Hence the theorem is valid for $-1<\sigma<0$. By meromorphic continuation, we obtain the result.
\end{proof}

\section*{Acknowledgements}
This work was supported by JSPS KAKENHI Grant Numbers JP20K03523 and JP19K14511.


\end{document}